\documentclass[12pt]{article}

\usepackage[usenames]{color}
\usepackage{amssymb}
\usepackage{graphicx}
\usepackage{amscd}

\usepackage[colorlinks=true,
linkcolor=webgreen,
filecolor=webbrown,
citecolor=webgreen]{hyperref}

\definecolor{webgreen}{rgb}{0,.5,0}
\definecolor{webbrown}{rgb}{.6,0,0}

\usepackage{color}
\usepackage{fullpage}
\usepackage{float}

\usepackage{graphics,amsmath,amssymb}
\usepackage{amsthm}
\usepackage{amsfonts}
\usepackage{latexsym}

\def\dis{\displaystyle}
\def\Enn{\mathbb{N}}

\author{Dzmitry Badziahin \footnote{Research supported by EPSRC  Grant EP/L005204/1}\\
Department of Mathematical Sciences\\
Durham University \\
Lower Mountjoy \\
Stockton Rd \\
Durham, DH1 3LE \\
United Kingdom \\
{\tt  dzmitry.badziahin@durham.ac.uk} \\
\and
Jeffrey Shallit \\
School of Computer Science \\
University of Waterloo \\
Waterloo, ON  N2L 3G1 \\
Canada \\
{\tt shallit@cs.uwaterloo.ca}
}

\title{An Unusual Continued Fraction}

\begin{document}

\theoremstyle{plain}
\newtheorem{theorem}{Theorem}
\newtheorem{corollary}[theorem]{Corollary}
\newtheorem{lemma}[theorem]{Lemma}
\newtheorem{proposition}[theorem]{Proposition}

\theoremstyle{definition}
\newtheorem{definition}[theorem]{Definition}
\newtheorem{example}[theorem]{Example}
\newtheorem{conjecture}[theorem]{Conjecture}

\theoremstyle{remark}
\newtheorem{remark}[theorem]{Remark}

\maketitle

\begin{abstract}
We consider the real number $\sigma$ with
continued fraction expansion $[a_0, a_1, a_2,\ldots] =
[1,2,1,4,1,2,1,8,1,2,1,4,1,2,1,16,\ldots]$,
where $a_i$ is the largest power of $2$ dividing $i+1$.
We compute the irrationality measure of $\sigma^2$ and
demonstrate that $\sigma^2$ (and $\sigma$) are both
transcendental numbers.  We also show that certain partial quotients
of $\sigma^2$ grow doubly exponentially, thus confirming a conjecture
of Hanna and Wilson.
\end{abstract}

\section{Introduction}

By a {\it continued fraction\/} we mean an expression of the
form
\begin{eqnarray*}
&&a_0+\frac{1}{\dis
        a_1+\frac{1}{a_2+\cdots +\dis \frac{1}{\dis a_n}}}
    \label{cfeq}
\end{eqnarray*}
or
\begin{eqnarray*}
&&a_0+\frac{1}{\dis
        a_1+\frac{1}{a_2+\cdots +\dis \frac{1}{\dis a_n + \cdots}}}
\label{cfeq2}
\end{eqnarray*}
where $a_1, a_2, \ldots, $ are positive integers and $a_0$ is an integer.
To save space, as usual, we write
$[a_0, a_1, \ldots, a_n]$ for the first
expression and
$[a_0, a_1, \ldots, ]$ for the second.  For properties
of continued fractions, see, for example,
\cite{Hardy&Wright:1985,Borwein:2014}.

It has been known since Euler and Lagrange that a real number has an
ultimately periodic continued fraction expansion if and only if it
is a quadratic irrational.  But the expansions of some other
``interesting'' numbers are also known explicitly.  For example
\cite{Walters:1968},
\begin{eqnarray*}
e &=& [2, (1,2n,1)_{n=1}^\infty] = [2,1,2,1, 1,4,1,1,6,1,1,8,1, \ldots ] \\
e^2 &=& [7,(3n-1,1,1,3n,12n+6)_{n=1}^\infty] = [7,2,1,1,3,18,5,1,1,6,30,\ldots] \\
\tan 1 &=&  [(1,2n-1)_{n=1}^\infty] = [1,1,1,3,1,5,1,7,\ldots] \\
\end{eqnarray*}
These three are examples of ``Hurwitz continued fractions'', where there is a
``quasiperiod'' of terms that grow linearly (see, for example,
\cite{Hurwitz:1896,Matthews&Walters:1970} and \cite[pp.~110--123]{Perron:1954}).
By contrast, no simple pattern is known for the expansions
of $e^3$ or $e^4$.

Recently there has been some interest in understanding the Diophantine
properties of
numbers whose continued fraction
expansion is generated by a simple computational model, such
as a finite automaton.   One famous example is the Thue-Morse sequence
on the symbols $\lbrace a, b \rbrace$ where $\overline{a} = b$ and
$\overline{b} = a$ is given by
$$ {\bf t} = t_0 t_1 t_2 \cdots = abbabaab \cdots$$
and is defined by
$$ t_n = \begin{cases}
    a, & \text{if $n = 0$}; \\
    t_{n/2}, & \text{if $n$ even}; \\
    \overline{t_{n-1}} , & \text{if $n$ odd}.
    \end{cases}
$$
Queff\'elec \cite{Queffelec:1998} proved that if $a, b$ are distinct
positive integers, then the real number $[{\bf t}] = [t_0, t_1, t_2, \ldots ]$
is transcendental.  Later, a simpler proof was found by
Adamczewski and Bugeaud \cite{Adamczewski&Bugeaud:2007}.
Queff\'elec \cite{Queffelec:2000} also proved the transcendence of a
much wider class of automatic continued fractions.    More
recently, in a series of papers, several authors explored the transcendence
properties of automatic, morphic, and Sturmian continued fractions
\cite{Allouche&Davison&Queffelec&Zamboni:2001,Adamczewski&Bugeaud:2005,Adamczewski&Allouche:2007,Adamczewski&Bugeaud:2010,Bugeaud:2012,Bugeaud:2013}.

All automatic sequences (and the more general class of morphic sequences)
are necessarily bounded.  A more general class, allowing unbounded
terms, is the $k$-regular sequences
of integers, for integer $k \geq 2$.
These are sequences $(a_n)_{n \geq 0}$ where the $k$-kernel, defined
by
$$ \lbrace (a_{k^e n + i})_{n \geq 0} \ : e \geq 0, 0 \leq i < k^e \rbrace,$$
is contained in a finitely generated module \cite{Allouche&Shallit:1992,Allouche&Shallit:2003}.  We state the following conjecture:

\begin{conjecture}
Every continued fraction where the terms
form a $k$-regular sequence of positive integers is transcendental or
quadratic.
\end{conjecture}

In this paper we study a particular example of a $k$-regular sequence:
$$ {\bf s} = s_0 s_1 s_2 \cdots = (1,2,1,4,1,2,1,8, \ldots) $$
where $s_i = 2^{\nu_2(i+1)}$ and $\nu_p(x)$ is the $p$-adic
valuation of $x$ (the exponent of the largest power of $p$ dividing
$x$).  To see that ${\bf s}$ is $2$-regular, notice that every
sequence in the $2$-kernel is a linear combination of ${\bf s}$
itself and the constant sequence $(1,1,1,\ldots)$.

The
corresponding real number $\sigma$ has continued fraction expansion
$$ \sigma = [{\bf s}] = [s_0, s_1, s_2, \ldots] = [1,2,1,4,1,2,1,8,\ldots]
= 1.35387112842988237438889 \cdots .$$
The sequence $\bf s$ is sometimes called the ``ruler sequence'', and is
sequence A006519 in Sloane's {\it Encyclopedia of Integer Sequences}
\cite{Sloane}.  The decimal expansion of $\sigma$ is sequence
A100338.

Although $\sigma$ has slowly growing partial quotients (indeed,
$s_i \leq i+1$ for all $i$), empirical calculation for $\sigma^2 =
1.832967032396003054427219544210417324 \cdots$
demonstrates the appearance
of some exceptionally large partial quotients.  For
example, here are the first few terms:
\begin{multline*}
\sigma^2 = [1, 1, 4, 1, 74, 1, 8457, 1, 186282390, 1, 1, 1, 2, 1, 430917181166219, \\
11, 37, 1, 4, 2, 41151315877490090952542206046,
11, 5, 3, 12, 2, 34, 2, 9, 8, 1, 1, 2, 7, \\
13991468824374967392702752173757116934238293984253807017, \ldots]
\end{multline*}
The terms of this continued fraction form sequence A100864 in
\cite{Sloane}, and were apparently first noticed by Paul D. Hanna
and Robert G. Wilson in November 2004.  The very large terms
form sequence A100865 in \cite{Sloane}.  In this note, we explain
the appearance of these extremely large partial quotients. The
techniques have some similarity with those of Maillet
(\cite{Maillet:1906} and \cite[\S 2.14]{Borwein:2014}).

Throughout the paper we use the following conventions.  Given a real
irrational number $x$ with partial quotients
$$x = [a_0, a_1, a_2, \ldots ]$$
we define the sequence of convergents by
\begin{align*}
p_{-2} = 0  & \quad & p_{-1} = 1 & \quad & p_n = a_n p_{n-1} + p_{n-2} & \quad (n \geq 0) \\
q_{-2} = 1  & \quad & q_{-1} = 0 & \quad & q_n = a_n q_{n-1} + q_{n-2} & \quad (n \geq 0) \\
\end{align*}
and then
$$[a_0, a_1, \ldots, a_n] = {{p_n} \over {q_n}}.$$

The basic idea of this paper is to use the following classical
estimate:
\begin{equation}\label{eq_approx}
\frac{1}{(a_{n+1}+2)q_n^2}<\left|x - \frac{p_n}{q_n}\right|
<\frac{1}{a_{n+1}q_n^2}.
\end{equation}
Therefore, in order to show that some partial quotients of $x$ are
huge, it is sufficient to find convergents $p_n/q_n$ of $x$
such that $|x-p_n/q_n|$ is much smaller than $q_n^{-2}$. We quantify
this idea in Section~\ref{sec_3}.

Furthermore, we use the Hurwitz-Kolden-Frame representation of continued
fractions \cite{Hurwitz&Kritikos:1986,Kolden:1949,Frame:1949} via $2 \times 2$ matrices, as follows:
\begin{equation}
M(a_0, \ldots, a_n) :=
\left[ \begin{array}{cc} a_0 & 1 \\ 1 & 0  \end{array} \right]
\left[ \begin{array}{cc} a_1 & 1 \\ 1 & 0  \end{array} \right]
\cdots
\left[ \begin{array}{cc} a_n & 1 \\ 1 & 0  \end{array} \right] =
\left[ \begin{array}{cc} p_n & p_{n-1} \\ q_n & q_{n-1}  \end{array}\right] .
\label{hur}
\end{equation}

By taking determinants we immediately deduce the classical identity
\begin{equation}
p_n q_{n-1} - p_{n-1} q_n = (-1)^{n+1}
\label{unit}
\end{equation}
for $n \geq 0$.

Given a finite sequence $z = (a_0, \ldots, a_n)$ we let $z^R$ denote
the reversed sequence $(a_n, \ldots, a_0)$.  A sequence is a {\it
palindrome} if $z = z^R$.  By taking the transpose of
Eq.~\eqref{hur} it easily follows that
\begin{equation}
M(a_n, \ldots, a_0) :=
\left[ \begin{array}{cc} a_n & 1 \\ 1 & 0  \end{array} \right]
\cdots
\left[ \begin{array}{cc} a_1 & 1 \\ 1 & 0  \end{array} \right]
\left[ \begin{array}{cc} a_0 & 1 \\ 1 & 0  \end{array} \right] =
\left[ \begin{array}{cc} p_n & q_n \\ p_{n-1} & q_{n-1}  \end{array}\right] .
\label{hur2}
\end{equation}
Hence if
$$[a_0, a_1, \ldots, a_n] = p_n/q_n$$
then
$$[a_n, \ldots, a_1, a_0] = p_n/p_{n-1}.$$

We now briefly mention ultimately periodic continued fractions.
By an expression of the form $[x, \overline{w}]$, where $x$ and $w$
are finite strings, we mean the continued fraction
$[x, w,w, w, \ldots]$, where the overbar or ``vinculum'' denotes the
repeating portion.  Thus, for example,
$$ \sqrt{7} = [2, \overline{1,1,1,4} \, ] = [2,1,1,1,4,1,1,1,4,1,1,1,4, \ldots].$$

We now recall a classical result.

\begin{lemma}
Let $a_0$ be a positive integer and $w$ denote a finite palindrome of
positive integers.  Then there exist positive integers $p, q$ such that
$$ [a_0, \overline{w, 2a_0} \,] = \sqrt{{p \over q}}.$$
\label{sqrt}
\end{lemma}

\begin{proof}
Define $y := [a_0, \overline{w, 2a_0} \,]$.  Then
$y = [a_0, w, a_0+y]$.  Letting
$$M(w) = \left[ \begin{array}{cc} \alpha & \beta \\
    \gamma & \delta \end{array}\right],$$
the corresponding matrix representation for $y$ is
$$
\left[ \begin{array}{cc} a_0 & 1 \\ 1 & 0 \end{array} \right]
\left[ \begin{array}{cc} \alpha & \beta \\ \gamma & \delta \end{array} \right]
\left[ \begin{array}{cc} a_0+y & 1 \\ 1 & 0 \end{array} \right]
=
\left[ \begin{array}{cc} (a_0\alpha+\gamma)(a_0+y) + a_0\beta + \delta & a_0
    \alpha + \gamma \\
    \alpha(a_0+y) + \beta & \alpha \end{array} \right].
$$
Since $w$ is a palindrome, it follows from Eq.~\eqref{hur} and \eqref{hur2}
that $\beta = \gamma$.  Hence
$$y = { {(a_0 \alpha + \beta)(a_0+y) + a_0\beta + \delta} \over
    {\alpha(a_0+y) + \beta}} .$$
Solving for $y$, which is clearly positive, we have
$$ y = \sqrt{ {p \over q}}$$
where $p = a_0^2 \alpha + 2 a_0 \beta + \delta$ and
$q = \alpha$, as desired.
\end{proof}

\section{Three sequences}

We now define three related sequences for $n \geq 2$:
\begin{eqnarray*}
u(n) &=& (s_1, s_2, \ldots, s_{2^n - 3}) \\
v(n) &=& (s_1, s_2, \ldots, s_{2^n - 2}) = (u(n),1) \\
w(n) &=& (s_1, s_2, \ldots, s_{2^n - 3}, 2) = (u(n), 2) .
\end{eqnarray*}
The following table gives the first few values of these quantities:

\begin{center}
\begin{tabular}{cccc}
$n$ & $u(n)$ & $v(n)$ & $w(n)$  \\
\hline
2 & 2 & 21 & 22  \\
3 & 21412 & 214121 & 214122  \\
4 & 2141218121412 & 21412181214121 & 21412181214122  \\
\end{tabular}
\end{center}

The following proposition, which is easily proved by induction,
gives the relationship between these sequences, for $n \geq 2$:
\begin{proposition}
\mbox{}
\begin{itemize}
\item[(a)] $u(n+1) = (v(n), 2^n, v(n)^R)$;
\item[(b)] $u(n)$ is a palindrome;
\item[(c)] $v(n+1) = (v(n), 2^n, 1, v(n))$.
\end{itemize}
\label{nice}
\end{proposition}

Furthermore, we can define the sequence of associated matrices with
$u(n)$ and $v(n)$:
\begin{eqnarray*}
M(u(n)) & := & \left[ \begin{array}{cc}
    c_n & e_n \\
    d_n & f_n \end{array} \right] \\
M(v(n)) & :=  & \left[ \begin{array}{cc}
    w_n & y_n \\
    x_n & z_n \end{array} \right] .
\end{eqnarray*}

The first few values of these arrays are given in the following
table.  As $d_n = e_n = z_n$ and $c_n = y_n$ for $n \geq 2$,
we omit the duplicate values.

\begin{center}
\begin{tabular}{c|cccccc}
$n$ & $c_n$ & $d_n$ & $f_n$ & $w_n$ & $x_n$  \\
\hline
2 & 2 &  1 & 0 & 3 & 1 & \\
3 & 48  & 17 & 6 & 65 & 23 & \\
4 & 40040 & 14169 & 5014 & 54209 & 19183 \\
5 & 51358907616 & 18174434593 & 6431407678  &
     69533342209 & 24605842271 \\
\end{tabular}
\end{center}

If we now define
$$ \sigma_n = [1, \overline{w(n)}\, ] $$
then Lemma~\ref{sqrt} with $a_0 = 1$ and $w = u(n)$ gives
$$ \sigma_n = \sqrt{ {{c_n + 2e_n + f_n} \over {c_n}}} .$$

Write $\sigma = [s_0, s_1, \ldots]$ and
$[s_0, s_1,\ldots, s_n] = {{p_n} \over {q_n}}$.
Furthermore define
$\hat{\sigma}_n = [1, u(n)]$.  Notice that
$\sigma$, $\sigma_n$, and $\hat{\sigma}_n$ all agree on the
first $2^n - 2$ partial quotients.  We have
$$\left| \sigma - \hat{\sigma}_n \right| < {1 \over{q_{2^n-3} q_{2^n - 2}}}$$
by a classical theorem on continued fractions
(e.g., \cite[Theorem 171]{Hardy&Wright:1985}), and furthermore,
since $s_{2^n - 3} = 2$, $s_{2^n-2} = 1$, we have, for $n \geq 3$, that
$$ \sigma < \sigma_n < \hat{\sigma}_n.$$
Hence
$$\left| \sigma - \sigma_n \right| < {1 \over{q_{2^n-3} q_{2^n - 2}}}.$$

Now by considering
$$
M(1) M(u(n)) M(1) = \left[ \begin{array}{cc} 1  & 1 \\ 1 & 0
\end{array} \right] \left[ \begin{array}{cc} c_n  & e_n \\ d_n & f_n
\end{array} \right] \left[ \begin{array}{cc} 1  & 1 \\ 1 & 0
\end{array} \right] = \left[ \begin{array}{cc} 2c_n+d_n  & e_n+f_n
\\ c_n+d_n & c_n \end{array} \right],
$$
we see that $q_{2^n-3} = c_n$ and $q_{2^n-2} = c_n + d_n$. For
simplicity write $g_n = c_n + 2e_n + f_n$. Then
$$\left| \sigma - \sigma_n \right|
= \left| \sigma - \sqrt{{g_n} \over {c_n}} \right| < {1 \over {c_n^2}},$$
and so
$$ \left| \sigma^2 - {{g_n} \over {c_n}} \right| =
\left| \sigma - \sqrt{{g_n} \over {c_n}} \right| \cdot
\left| \sigma + \sqrt{{g_n} \over {c_n}} \right| < {3 \over {c_n^2}}.$$
So we have already found good approximations
of $\sigma^2$ by rational numbers.  In the next section we will show that
$g_n$ and $c_n$ have a large common factor, which will improve the
quality of the approximation.

\section{Irrationality measure of $\sigma^2$}\label{sec_3}

From Proposition~\ref{nice} (a), we get that the matrix
$$
\left[ \begin{array}{cc} c_{n+1}
& e_{n+1} \\ d_{n+1} & f_{n+1} \end{array} \right ]
$$
associated with
$u(n+1)$ is equal to the matrix
associated with $(v(n), 2^n, v(n)^R)$, which is
$$
\left[ \begin{array}{cc} w_n & y_n \\ x_n & z_n \end{array} \right ]
\left[ \begin{array}{cc} 2^n & 1 \\ 1 & 0 \end{array} \right ]
\left[ \begin{array}{cc} w_n & x_n \\ y_n & z_n \end{array} \right ]
=
\left[ \begin{array}{cc} 2^n w_n^2 + 2w_n y_n  & 2^n w_n x_n + x_n y_n + w_n z_n \\
2^n w_n x_n + x_n y_n + w_n z_n & 2^n x_n^2 + 2x_n z_n \end{array} \right ].
$$

Notice that
\begin{equation}
c_{n+1} = (2^n w_n + 2y_n) w_n.
\label{cn}
\end{equation}
On the other hand, we have
\begin{eqnarray*}
g_{n+1} &=& c_{n+1} + 2d_{n+1} + f_{n+1} \\
&=&  c_{n+1} + 2 (2^n w_n x_n + x_n y_n + w_n z_n) +  2^n x_n^2 + 2x_n z_n \\
&=&  c_{n+1} + 2 (2^n w_n x_n + x_n y_n + w_n z_n) +  2^n x_n^2 + 2x_n z_n
    + 2 ( x_n y_n - w_n z_n + 1) \\
&=&  c_{n+1} + (2^n w_n + 2y_n) 2 x_n + 2^n x_n^2 + 2x_n z_n + 2  \\
&=& (2^n w_n + 2y_n) (2x_n + w_n) + 2^n x_n^2 + 2x_n z_n + 2 ,
\end{eqnarray*}
where we have used Eq.~\eqref{unit}.
By Euclidean division, we get
$$\gcd(g_{n+1}, 2^n w_n + 2y_n) =
\gcd(2^n w_n + 2y_n, 2^n x_n^2 + 2x_n z_n + 2).$$

Next, we interpret Proposition~\ref{nice} (c) in terms of matrices.
We get that the matrix
$$
\left[ \begin{array}{cc} w_{n+1}
& y_{n+1} \\ x_{n+1} & z_{n+1} \end{array} \right ]
$$
associated with
$v(n+1)$ is equal to the matrix
associated with $(v(n), 2^n, 1, v(n))$, which is
\begin{eqnarray}
&& \left[ \begin{array}{cc} w_n & y_n \\ x_n & z_n \end{array} \right]
\left[ \begin{array}{cc} 2^n & 1 \\ 1 & 0 \end{array} \right]
\left[ \begin{array}{cc} 1 & 1 \\ 1 & 0 \end{array} \right]
\left[ \begin{array}{cc} w_n & y_n \\ x_n & z_n \end{array} \right] \nonumber\\
&=&
\left[ \begin{array}{cc}   (2^n+1)w_n^2 + 2^n w_n x_n + y_n(w_n+x_n) &
(2^n + 1) w_n y_n + 2^n w_n z_n + y_n(y_n + z_n)\\
(2^n + 1) x_n w_n + 2^n x_n^2 + z_n (w_n + x_n) &
(2^n + 1) x_n y_n + 2^n x_n z_n + z_n (y_n + z_n)
\end{array} \right] . \label{eq_wn}
\end{eqnarray}

Letting $r_n := 2^n (w_n + x_n) + w_n + y_n + z_n$, we see that
\begin{eqnarray}
2^{n+1} w_{n+1} + 2y_{n+1}
&=& 2^{n+1} (  (2^n+1)w_n^2 + 2^n w_n x_n + y_n(w_n+x_n)) + \nonumber \\
&& \quad\quad 2 ( (2^n + 1) w_n y_n + 2^n w_n z_n + y_n(y_n + z_n) ) \nonumber \\
&=& 2(2^n w_n + y_n)(2^n (w_n + x_n) + w_n + y_n + z_n)  \nonumber \\
&=& 2 (2^n w_n + y_n) r_n.   \label{nn}
\end{eqnarray}

Now
\begin{eqnarray*}
x_{n+1} &=&  (2^n + 1) x_n w_n + 2^n x_n^2 + z_n (w_n + x_n) \\
&=& x_n (2^n (w_n + x_n) + w_n + y_n + z_n) + w_nz_n - x_n y_n \\
&=& x_n r_n + 1
\end{eqnarray*}
and
\begin{eqnarray*}
z_{n+1} &=& (2^n + 1) x_n y_n + 2^n x_n z_n + z_n (y_n + z_n) \\
&=& z_n (2^n (w_n + x_n) + w_n + y_n + z_n) + (2^n+1)(x_n y_n - w_n z_n) \\
&=& z_n r_n - 2^n - 1.
\end{eqnarray*}

It now follows, from some tedious algebra, that
\begin{multline}
{{2^n x_{n+1}^2 + x_{n+1} z_{n+1} + 1} \over
    {r_n}} =
(2^n + 1) w_n x_n z_n + 2^n(2^n + 1)w_n x_n^2 + z_n + (2^n-1)x_n  \\
    + 2^n x_n^2 y_n + 2^{n+1} x_n^2 z_n + x_n y_n z_n +
    2^{2n} x_n^3 + x_n z_n^2 .
\label{tedious}
\end{multline}

From Eq.~\eqref{cn} and reindexing, we get
\begin{eqnarray*}
c_{n+2} &=& w_{n+1} (2^{n+1} w_{n+1} + 2 y_{n+1}) \\
&=& 2 w_{n+1} (2^n w_n + y_n) r_n,
\end{eqnarray*}
where we used Eq.~\eqref{nn}.    Also, from the argument above
about gcd's and Eq.~\eqref{tedious}, we see that $2 r_n \ | \ g_{n+2}$.
Hence for $n \geq 2$ we have
$$ {{g_{n+2}}\over {c_{n+2}}} = {{P_{n+2}} \over {Q_{n+2}}} $$
for integers $P_{n+2} := {{g_{n+2}} \over {2r_n}}$ and
$Q_{n+2} := w_{n+1} (2^n w_n + y_n)$.
It remains to see
that $P_{n+2}/Q_{n+2}$ are particularly good rational
approximations to $\sigma^2$.

Since $w_n/x_n$ and $y_n/z_n$ denote successive convergents
to a continued fraction, we clearly have
$w_n \geq x_n$, $w_n \geq y_n$, and $w_n \geq z_n$.  It follows
that
\begin{eqnarray*}
Q_{n+2} &=& w_{n+1} (2^n w_n + y_n) \\
&=&  ((2^n+1)w_n^2 + 2^n w_n x_n + y_n(w_n+x_n)) (2^n w_n + y_n)  \\
&\leq & (2^{n+1} +3) w_n^2 \cdot (2^n + 1) w_n  \\
&=& (2^{n+1} + 3)(2^n + 1)w_n^3.
\end{eqnarray*}
On the other hand,
\begin{eqnarray*}
c_{n+2} &=& 2 Q_{n+2} r_n  \\
&>& 2(2^n+1) w_n^2 \cdot 2^n w_n\cdot (2^n+1) w_n  = 2^{n+1} (2^n+1)^2 w_n^4 \\
&\geq & Q_{n+2}^{4/3} {{2^{n+1} (2^n + 1)^2} \over
    {((2^{n+1} + 3)(2^n + 1))^{4/3}}} \\
& > & Q_{n+2}^{4/3}.
\end{eqnarray*}
This gives
$$ \left| \sigma^2 - {{P_{n+2}} \over {Q_{n+2}}}  \right| < Q_{n+2}^{-8/3}$$
for all integers $n \geq 2$.

The result we have just shown can be nicely formulated in terms of
the irrationality measure. Recall that the {\it irrationality
measure} of a real number $x$ is defined to be the infimum, over all
real $\mu$, for which the inequality
$$ \left| x - {p \over q} \right| < {1 \over {q^\mu}} $$
is satisfied by at most finitely many integer pairs $(p,q)$.

\begin{theorem}\label{th_1}
The irrationality measure of $\sigma^2$ is at least $8/3$.
\end{theorem}

Note that the classical Khintchine theorem (e.g.,
\cite[Chapter VII, Theorem I]{Cassels:1957})
states that for almost
all real numbers (in terms of Lebesgue measure), the
irrationality exponent equals two. Hence Theorem~\ref{th_1} says
that the number $\sigma^2$ belongs to a very tiny set of zero
Lebesgue measure.

Furthermore, the famous Roth theorem \cite{Roth:1955} states that
the irrationality exponent of every irrational algebraic number is
two. Therefore we conclude that $\sigma^2$ (and hence $\sigma$) are
transcendental numbers.

We now provide a lower bound for some very
large partial quotients of $\sigma^2$. For each $n\ge 2$ we certainly
have
$$
\left|\sigma - \frac{P_{n+2}}{Q_{n+2}}\right| <Q^{-8/3}_{n+2}
<\frac{1}{2Q_{n+2}^2}.
$$
In particular this implies that the rational
number $P_{n+2}/Q_{n+2}$ is a convergent of
$\sigma^2$.

Notice that $P_{n+2}$ and $Q_{n+2}$ are
not necessarily relatively prime.
Let $\tilde{P}_{n+2}/\tilde{Q}_{n+2}$ denote the reduced
fraction of $P_{n+2}/Q_{n+2}$.
If $\tilde{P}_{n+2}/\tilde{Q}_{n+2}$ is the $m$'th convergent
of $\sigma^2$, then define $A_{n+2}$ to be the $(m+1)$'th partial
quotient of $\sigma^2$.
Then the estimate~\eqref{eq_approx}
implies
$$
\frac{1}{(A_{n+2}+2)\tilde{Q}_{n+2}^2} <
\left|\sigma^2-\frac{\tilde{P}_{n+2}}{\tilde{Q}_{n+2}} \right| <
\frac{3}{c_{n+2}^2} \le \frac{3}{4r_n^2\tilde{Q}_{n+2}^2},
$$
Hence $A_{n+2} \ge 4r_n^2 - 2$.

From the formula for $r_n$ and the inequalities $w_n\ge x_n, w_n\ge
y_n, w_n\ge z_n$ one can easily derive
$$
(2^n+1)w_n\le r_n\le (2^{n+1}+3)w_n.
$$
This, together with the formula~\eqref{eq_wn} for $w_{n+1}$, gives
the estimate
$$
r_{n+1} \ge (2^{n+1}+1)w_{n+1}\ge (2^{n+1}+1)(2^n+1)w_n^2 > r_n^2+1.
$$
Therefore the sequence $4r^2_n-2$, and in turn $A_{n+2}$, grow doubly
exponentially. This phenomenon explains the observation of Hanna and
Wilson for the sequence A100864 in~\cite{Sloane}.

The first few values of the sequences we have been discussing are
given below:
\begin{center}
\begin{tabular}{c|ccccccc}
$n$ & $\sigma_n^2$ & $\hat{\sigma}_n$ & $g_n$ & $r_n$ & $P_n$ & $Q_n$ & $A_n$ \\
\hline
3 & ${{11} \over 6}$ & ${{65}\over{48}}$ & 88 & 834 & 11 & 6 & 74 \\
\ \\
4 & ${{834}\over {455}}$ & ${{54209} \over {40040}}$ & 73392 & 1282690 & 1668 & 910 & 8457 \\
\ \\
5 & ${{7054795}\over {3848839}}$ & ${{69533342209} \over {5135807616}}$ & 94139184480 & 3151520587778 & 56438360 & 30790712 & 186282390
\end{tabular}
\end{center}

\section{Additional remarks}

The same idea can be used to bound the irrationality exponent of an
infinite collection of numbers $\sigma = [s_0,s_1,s_2,\ldots,]$.
Indeed, there is nothing particularly special about the terms $2^n$
appearing in Proposition~\ref{nice}. One can check that the same
result holds if the strings $u(n)$ and $v(n)$ satisfy the following
modified
properties from Proposition~\ref{nice} for infinitely many numbers
$n\ge 2$:
\begin{itemize}
\item[(a')] $u(n+1) = (v(n), k_n, v(n)^R)$;
\item[(b')] $u(n+2)$ is a palindrome;
\item[(c')] $v(n+2) = (v(n+1), 2k_n, 1, v(n+1))$.
\end{itemize}
In particular one can easily check these properties for a string $
{\bf s} = s_0s_1s_2\cdots$ such that $s_i = f(\nu_2(i+1))$ where $f:
\Enn \rightarrow \Enn$ is a function satisfying the following
conditions:
\begin{enumerate}
\item $f(0) = 1$;
\item $f(n+1) = 2f(n)$ for infinitely many positive integers $n$.
\end{enumerate}

%
%
%
%

\end{document}